\documentclass[12pt,a4paper]{amsart}
\usepackage{amsfonts,amsmath,amssymb,amscd}
\usepackage{fancybox}
\usepackage[latin1]{inputenc}
\usepackage[usenames]{color}
\usepackage{graphicx}
\usepackage{subfig}
\usepackage{float}
\usepackage{fullpage}
\usepackage{epsfig,amssymb}

\theoremstyle{plain}
\newtheorem{thm}{Theorem}[section]

\newtheorem{lem}[thm]{Lemma}
\newtheorem{prop}[thm]{Proposition}

\theoremstyle{definition}

\newtheorem{rem}[thm]{Remark}

\def \r{\mbox{${\mathbb R}$}}

\def \h{\mbox{${\mathbb H}$}}

\def \s{\mbox{${\mathbb S}$}}

\DeclareMathOperator{\grad}{grad}
\DeclareMathOperator{\trace}{trace}

\DeclareMathOperator{\Res}{Res}
\DeclareMathOperator{\Ricci}{Ricci}
\DeclareMathOperator{\dive}{div}

\begin{document}
	
\title{On the existence of closed biconservative surfaces in space forms}

\author{S. Montaldo}
\address{Universit\`a degli Studi di Cagliari\\
Dipartimento di Matematica e Informatica\\
Via Ospedale 72\\
09124 Cagliari, Italia}
\email{montaldo@unica.it}

\author{A. P\'ampano}
\address{Department of Mathematics\\
University of the Basque Country\\
Aptdo. 644\\
48080, Bilbao, Spain.}
\email{alvaro.pampano@ehu.eus}

\date{\today}

\begin{abstract}
Biconservative surfaces of Riemannian 3-space forms $N^3(\rho)$, are either constant mean curvature (CMC) surfaces or rotational linear Weingarten surfaces verifying the relation $3\kappa_1+\kappa_2=0$ between their principal curvatures $\kappa_1$ and $\kappa_2$. We characterise the profile curves of the non-CMC biconservative surfaces as the critical curves for a suitable curvature energy. Moreover, using this characterisation, we prove the existence of a discrete biparametric family of closed, i.e. compact without boundary, non-CMC biconservative surfaces in the round 3-sphere, $\mathbb{S}^3(\rho)$. However, none of these closed surfaces is embedded in $\mathbb{S}^3(\rho)$.
\end{abstract}

\thanks{Work partially supported by Fondazione di Sardegna and Regione Autonoma della Sardegna (Project GESTA). The second author has been partially supported by MINECO-FEDER grant PGC2018-098409-B-100, Gobierno Vasco grant IT1094-16 and Programa Predoctoral de Formaci\'on de Personal Investigador No Doctor del Gobierno Vasco, 2015. He also wants to thank the Department of Mathematics and Computer Science of the University of Cagliari for the warm hospitality during his stay.}


\maketitle

\section{Introduction}

A hypersurface $M^{n-1}$ in a $n$-dimensional Riemannian manifold $N^n$ is called {\it biconservative} if
\begin{equation*}
2S_\eta(\grad H)+ (n-1) H \grad H-2 H\Ricci(\eta)^{\top}=0
\end{equation*}
where $\eta$ is a unit normal vector field, $S_\eta$ is the shape operator, $H=\trace S_\eta/(n-1)$ is the mean curvature function and $\Ricci(\eta)^{\top}$ is the tangent component of the Ricci curvature of $N$ in the direction of $\eta$.

The notion of biconservative hypersurfaces was introduced in \cite{Caddeo-Montaldo-Oniciuc-Piu}, as we shall detail in the next section, where the authors classify, locally, biconservative surfaces into $3$-dimensional space forms (see also \cite{Fuannali}). 
More precisely, in \cite{Caddeo-Montaldo-Oniciuc-Piu}, it was proved that a biconservative surface of a 3-space form, $N^3(\rho)$, is either a CMC surface or a rotational surface. Moreover, in the same paper, a relation between the Gaussian curvature, $K$, and the mean curvature, $H$, of the non-CMC biconservative surfaces of $N^3(\rho)$ was stated. Indeed, these rotational surfaces verify
\begin{equation}
K=-3H^2+\rho\, . \label{relationcurvatures-intro}
\end{equation}
Throughout this paper, we are going to understand a \emph{Weingarten surface} as a surface of $N^3(\rho)$ where the two principal curvatures $\kappa_1$ and $\kappa_2$ satisfy a certain relation $\Phi(\kappa_1,\kappa_2)=0$. These surfaces were introduced by Weingarten in \cite{Weingarten} and its study occupies an important role in classical Differential Geometry.

The simplest relation $\Phi(\kappa_1,\kappa_2)=0$ is the pure linear relation, that is, 
\begin{equation}\kappa_1=a\kappa_2\,,\quad a\in\r\,. \label{linearrelation-intro}
\end{equation}
Rotational surfaces in Riemannian 3-space forms verifying the relation \eqref{linearrelation-intro} between their principal curvatures were geometrically described by Barros and Garay in \cite{Barros-Garay} where they gave a variational  characterisation of the parallels. On the other hand, in \cite{Lopez-Pampano}, the profile curve of rotational linear Weingarten surfaces of $\mathbb{R}^3$ was characterised as a critical curve for a {\it curvature energy problem}.

In our case, relation \eqref{relationcurvatures-intro} implies that non-CMC biconservative surfaces are linear Weingarten surfaces for $a=-1/3$ in \eqref{linearrelation-intro}, as it was first pointed by Fu and Li in \cite{Fu-Li}. Thus, we have the following description. 

\begin{prop}\label{LW} The non-CMC biconservative surfaces of a 3-dimensional space form $N^3(\rho)$ are rotational linear Weingarten surfaces verifying
\begin{equation}
3\kappa_1+\kappa_2=0\,, \label{LWrelationbiconservative}
\end{equation}
where $\kappa_1=-\kappa$ is minus the curvature of the profile curve. Moreover, let $S\subset N^3(\rho)$ be a rotational linear Weingarten surface verifying \eqref{LWrelationbiconservative}, then $S$ is a biconservative surface.
\end{prop}

Throughout this paper, we are going to consider the following \emph{bending-type energy problem}. More precisely,  we consider the  \emph{curvature energy functional} given by
\begin{equation}\label{curveture-energy}
\mathbf{\Theta}(\gamma):=\int_\gamma \kappa^{1/4}\, , 
\end{equation}
acting on the space of arc-length parametrized non-geodesic curves $\gamma:I\to N$, where $I$ is a real interval. Here, $\kappa$ denotes the curvature of $\gamma$. In the Euclidean 3-space, $\mathbb{R}^3$, this functional has been studied in \cite{Garay-Pampano}, where its critical curves have been used to produce solutions of a generalized Ermakov-Milne-Pinney equation. On the other hand, in the unit round 3-sphere, $\mathbb{S}^3(1)$, it was analysed in \cite{Arroyo-Garay-Mencia}. For more details about this functional see \S 3.

The first main result of the paper is the following characterisation.

\begin{thm}\label{variationalcharacterization} Let $S$ be a non-CMC biconservative surface of a 3-dimensional space form $N^3(\rho)$. Then, locally, $S$ is a rotational surface whose profile curve verifies the Euler-Lagrange equations for the functional \eqref{curveture-energy}.
\end{thm}

A converse of Theorem \ref{variationalcharacterization} is also true and gives us a way of constructing all non-CMC biconservative surfaces of 3-space forms, as it will be explained in Theorem \ref{converse}.

A natural problem is to investigate, using the variational characterisation, the existence of closed (i.e. compact without boundary) non-CMC biconservative surfaces into $N^3(\rho)$. This problem was raised for the first time in \cite{Nistor} and then in \cite{FLO,FNO}. In \S 5 we tackle this problem and solve it. We first prove, in Proposition~\ref{non-existence-closed}, that there are no closed non-CMC biconservative surfaces in 3-space forms, $N^3(\rho)$, with $\rho\leq 0$. While, for the case of $\mathbb{S}^3(\rho)$, we prove the following existence result.

\begin{thm}\label{complete-case} There exists a discrete biparametric family of closed non-CMC biconservative surfaces in the round 3-sphere, $\mathbb{S}^3(\rho)$. However, there are no closed non-CMC biconservative surfaces embedded in $\mathbb{S}^3(\rho)$.
\end{thm}

\section{Biharmonic Maps and Biconservative Immersions}

\emph{Harmonic maps} $\varphi\,: (M,g)\rightarrow (N,h)$ between Riemannian manifolds are the critical points of the energy functional
\begin{equation}
E(\varphi)=\frac{1}{2}\int_M \lvert d\varphi \rvert ^2 v_g\, . \nonumber
\end{equation}
Their corresponding Euler-Lagrange equation is given by the vanishing of the tension field
\begin{equation}
\tau(\varphi)=\trace \nabla d\varphi\,. \label{tensionfield}
\end{equation}
In \cite{Eells-Sampson}, Eells and Sampson suggested to study \emph{biharmonic maps}, which are the critical points of the \emph{bienergy} functional
\begin{equation}
E_2(\varphi)=\frac{1}{2}\int_M \lvert \tau(\varphi)\rvert ^2 v_g\, . \label{biharmonicenergy}
\end{equation} 
The first variation formula of the bienergy was derived by Jiang, \cite{Jiang-1}. Moreover, he showed that the Euler-Lagrange equation for $E_2$ is 
\begin{equation}
\tau_2(\varphi)=-J(\tau(\varphi))=-\Delta \tau(\varphi)-\trace R^N(d\varphi,\tau(\varphi))d\varphi=0\, , \label{bitensionfield}
\end{equation}
where $J$ is the Jacobi operator of $\varphi$. The curvature operator of $(N,h)$ is denoted by $R^N$ and it can be computed as
\begin{equation}
R^N(X,Y)=\nabla_X\nabla_Y-\nabla_Y\nabla_X-\nabla_{[X,Y]}\, , \nonumber
\end{equation}
for any vector fields $X$ and $Y$ in $N$, where $\nabla$ denotes the Levi-Civita connection. Finally, the symbol $\Delta$ in \eqref{bitensionfield} represents the rough Laplacian on sections $\varphi^{-1}(TN)$, which, for a local orthonormal frame $\{e_i\}_{i=1}^m$ on $M$, is defined by
\begin{equation}
\Delta=-\sum_{i=1}^m\left(\nabla_{e_i}^\varphi\nabla_{e_i}^\varphi-\nabla^\varphi_{\nabla_{e_i}^M e_i}\right). \nonumber
\end{equation}
The equation $\tau_2(\varphi)=0$ is called the \emph{biharmonic equation}. Since the Jacobi operator $J$ is linear, it is easy to check that harmonic maps are always biharmonic.

If $\varphi:M\to N$ is an isometric immersion the decomposition of the bitension field with respect to its normal and
tangent components was obtained with contributions of \cite{BMO13,C84,LM08,O10} and for hypersurfaces it can be summarised in the following theorem.

\begin{thm}\label{decompbitension} Let $\varphi:M^{n-1}\to N^n$ be an isometric immersion with unit normal vector field $\eta$ and mean curvature vector field $\mathbf{H}=H\eta$. Then, the normal and tangential components of $\tau_2(\varphi)=0$ are respectively
\begin{eqnarray}
\Delta H+H \lvert S_\eta \rvert^2-H \Ricci(\eta,\eta)&=&0\, ,\nonumber\\
2S_\eta(\grad H)+ (n-1) H \grad H-2 H\Ricci (\eta)^T&=&0\, , \nonumber
\end{eqnarray}
where $S_\eta$ is the shape operator and $\Ricci(\eta)^T$ is the tangent component of the Ricci curvature of $N$ in the direction of the vector field $\eta$.
\end{thm}

\begin{rem}
Let $\gamma: I\rightarrow N$ be an arc-length parametrized curve from an open interval $I\subset \mathbb{R}$ to a Riemannian manifold $N$. In this case, putting $T=\gamma '$, the bienergy functional, \eqref{biharmonicenergy}, reduces to
\begin{equation}
E_2(\gamma)=\int_\gamma \kappa^2\, . \label{bienergycurves}
\end{equation}
Then, the Euler-Lagrange equation for the bienergy when it acts on the space of all maps between $I$ and $N$, can be written as 
\begin{equation}
\tau_2(\gamma)=\nabla_T^3 T+R(\nabla_T T, T)T=0\, ,\label{bitensionfieldcurves}
\end{equation}
and its solutions are called \emph{\it biharmonic curves}. Take into account that \emph{harmonic curves}, that is, arc-length parametrized solutions of the tension field $\tau(\gamma)$, \eqref{tensionfield}, are just \emph{geodesics}. Moreover, as mentioned before, harmonic maps are also biharmonic maps and, therefore, biharmonic curves represent a generalization of geodesics.

One can also study the bienergy functional $E_2$, \eqref{bienergycurves}, acting on the space of curves immersed in $N$. That is, in this case, the problem consists of seeking critical curves among arc-length parametrized curves and it is usually referred as \emph{bending energy problem}, while its critical curves are called \emph{elastic curves}.

We point out that the Euler-Lagrange equation for this last variational problem over curves is
\begin{equation}
\nabla_T^3 T+\frac{3}{2}\nabla_T \left(\kappa^2 T\right)-R(\nabla_T T, T)T=0 \, ,\nonumber
\end{equation}
which, in principle, is different from \eqref{bitensionfieldcurves}. However, geodesics are also elastic curves, which means that this is another way of generalizing the notion of a geodesic.
\end{rem}

\subsection{Biconservative Immersions}
As described by Hilbert in \cite{Hilbert}, the \emph{stress-energy} tensor associated with a variational problem is a symmetric 2-covariant tensor $\mathcal{S}$ which is conservative at critical points, that is, with $\dive\mathcal{S}=0$.

In the context of harmonic maps $\varphi:(M,g)\to (N,h)$ between two Riemannian manifolds the stress-energy tensor was studied in detail by Baird and Eells in \cite{Baird-Eells} (see also \cite{Baird-Ratto} and \cite{Sanini}). Indeed, the tensor 
\begin{equation}
\mathcal{S}=\frac{1}{2}\lvert d\varphi\rvert^2g-\varphi^* h\, \nonumber
\end{equation}
satisfies $\dive\mathcal{S}=-\langle \tau(\varphi), d\varphi\rangle$, where $\tau(\varphi)$ is given by \eqref{tensionfield}. Therefore, we have that $\dive\mathcal{S}=0$ when the map is harmonic.
Moreover, when $\varphi$ is any isometric immersion, the condition $\dive\mathcal{S}=0$ is always satisfied, since the tension field $\tau(\varphi)$ is normal to the submanifold.

The study of the stress-energy tensor for the bienergy \eqref{biharmonicenergy} was initiated in \cite{Jiang} and afterwards developed in \cite{Loubeau-Montaldo-Oniciuc}. Its expression is
\begin{equation}
\mathcal{S}_2 (X,Y)=\frac{1}{2}\lvert \tau(\varphi)\rvert^2\langle X, Y\rangle+\langle d\varphi, \nabla \tau(\varphi)\rangle\langle X,Y\rangle- \langle d\varphi(X), \nabla_Y \tau(\varphi)\rangle-\langle d\varphi(Y),\nabla_X\tau(\varphi)\rangle, \nonumber
\end{equation}
and it satisfies the condition
\begin{equation}
\dive\mathcal{S}_2=-\langle \tau_2(\varphi), d\varphi\rangle,\label{div2}
\end{equation}
where $\tau_2(\varphi)$ is the bitension field given in \eqref{bitensionfield}. Due to \eqref{div2}, we have that $\mathcal{S}_2$ is conforming to the principle of a stress-energy tensor for the bienergy.

Now, if $\varphi$ is an isometric immersion, \eqref{div2} reads
\begin{equation}
\left(\dive\mathcal{S}_2\right) ^{\sharp}=-\tau_2(\varphi)^{T}, \label{tangentpart}
\end{equation}
where $\sharp$ denotes the musical isomorphism sharp.

An isometric immersion is \emph{biconservative} if the corresponding stress-energy tensor $\mathcal{S}_2$ is conservative, that is, if $\dive\mathcal{S}_2=0$. From \eqref{tangentpart}, biconservative isometric immersions correspond to immersions with vanishing tangential part of the corresponding bitension field, that is, using Theorem \ref{decompbitension},  an isometric immersion $\varphi: M^{n-1}\to N^n$ is \emph{biconservative} if and only if $\varphi$ satisfies the condition
\begin{equation}
2S_\eta(\grad H)+ (n-1) H\, \grad H-2 H\Ricci (\eta)^T=0\, . \label{biconserveq}
\end{equation}
An hypersurface $M^{n-1}$ immersed in this way is usually called a \emph{biconservative hypersurface}.

Notice that, in the particular case when the ambient space is any space form of dimension $n$ and constant sectional curvature $\rho$, $N^n(\rho)$, the tangential part of the Ricci curvature vanishes, and therefore equation \eqref{biconserveq} simplifies to
\begin{equation}
2S_\eta(\grad H)+ (n-1) H\grad H=0\, .\nonumber
\end{equation}
The theory of biconservative hypersurfaces is developing very rapidly and we refer the reader to the papers \cite{FOP,YuFu,MOR,MOR2,Montaldo-Onnis-Passamani} and the references therein. 

\subsection{Invariant Surfaces}
We end this section recalling that  a surface $S$ into  $N^3(\rho)$ is said to be an \emph{invariant surface} if it stays invariant under the action of a one-parameter group of isometries of $N^3(\rho)$. The one-parameter group of isometries of $N^3(\rho)$ is determined by the flow of a Killing vector field of $N^3(\rho)$. Take $\Phi\in {\rm Isom}^+\left(N^3(\rho)\right)$ an orientation preserving isometry of $N^3(\rho)$ and assume that $\Phi$ is a rotation whose axis is a given geodesic $\zeta$. The group of all isometries in ${\rm Isom}^+\left( N^3(\rho)\right) $ with the same axis is isomorphic to $SO(2)\backsimeq \mathbb{S}^1$ and acts naturally on $N^3(\rho)$. A \emph{rotational surface}, $S\subset N^3(\rho)$, is an $SO(2)$-invariant surface, where $SO(2)$ is considered to be the subgroup of isometries, ${\rm Isom}^+\left(N^3(\rho)\right)$, acting as explained before. The group $SO(2)$ fixes all the points of the rotation axis $\zeta$ and rotates an everywhere orthogonal curve $\gamma$ (the \emph{profile curve}) around $\zeta$ sweeping out a rotational surface which will be denoted by $S_\gamma$ from now on.

\section{Bending-Type Curvature Energy}

Let us denote by $\gamma$ an arc-length parametrized curve immersed in $N^3(\rho)$. If $\gamma(s)$ is a unit speed non-geodesic smooth curve immersed
in $N^3(\rho)$, then $\gamma(s)$ is a \textit{Frenet curve} of
rank $2$ or $3$  and the  standard \textit{Frenet frame} along
$\gamma(s)$ is given by $\{T,N, B\}(s)$, where $N$ and $B$ are the
\emph{unit normal} and \emph{unit binormal} to the curve,
respectively, and $B$ is chosen so that ${\rm det}(T,N,B)=1$. Then the
\emph{Frenet equations}
\begin{eqnarray}
\nabla_T T(s)&=&\kappa(s) N (s)\, , \label{frenet1}\\
\nabla_T N(s)&=&-\kappa(s) T(s) +\tau(s) B(s)\, , \label{frenet2}\\
\nabla_T B(s)&=& -\tau (s)N (s)\, ,\label{frenet3}
\end{eqnarray}
define the \textit{curvature}, $\kappa(s)$ (we will always consider $\kappa(s)\geq 0$), and \textit{torsion}, $\tau(s)$, along $\gamma(s)$ (do not confuse the notation with the tension field $\tau(\varphi)$ defined in \eqref{tensionfield}).

In a Riemannian 3-space form any local geo\-me\-trical scalar
defined along Frenet curves can always be expressed as a function
of their curvatures and derivatives. Notice that, even if the rank
of $\gamma$ is $2$ (i.e, $\tau =0$), the binormal
$B=T\times N$ is still well defined and above
formulas \eqref{frenet1}-\eqref{frenet3} still make sense. Moreover, in a 3-space form, $N^3(\rho)$, a curve verifying $\tau=0$ can be assumed to lie in a totally geodesic surface $N^2(\rho)$. Curves whose torsion vanishes are called \emph{planar curves}. From now on, we are going to deal with planar curves, unless the opposite is said.

Let us consider the following \emph{curvature energy functional} \eqref{curveture-energy}
\begin{equation}
\mathbf{\Theta}(\gamma )=\int_{\gamma} \kappa^{1/4}=
\int_{0}^L  \kappa^{1/4}(s)\, ds \, ,\nonumber
\end{equation}
where, as usual, the arc-length or natural parameter is
represented by $s\in \left[0,L\right]$, $L$ being the length of
$\gamma$. Then, we consider $\mathbf{\Theta}$
acting on the following spaces of curves, satisfying given
boundary conditions in $(N^2(\rho),\langle\cdot,\cdot\rangle)$. We shall denote by $\Omega^{\rho}
_{p_op_1}$ the space of smooth immersed curves of $N^2(\rho)$,
joining two given points of it,
that is:
\begin{equation}
\Omega^{\rho}_{p_op_1}=\{\delta:[0,1]\rightarrow N^2(\rho)\, ;\, \delta (i)=p_i, i\in\{0,1\}, \frac{d\delta}{dt}(t)\neq 0,
\forall t\in [0,1]  \}, \label{omega}\nonumber
\end{equation}
\noindent where $p_i\in N^2(\rho), i\in\{0,1\}$, are arbitrary given points of $N^2(\rho)$.

For a curve $\gamma:[0,1]\rightarrow N^2(\rho)$, we
take a variation of $\gamma$,
$\Gamma=\Gamma(t,\bar{t}):[0,1]\times
(-\varepsilon,\varepsilon)\rightarrow N^2(\rho)$ with
$\Gamma(t,0)=\gamma(t)$. Associated to this variation we have the
vector field $W=W(t)=\frac{\partial\Gamma}{\partial \bar{t}}(t,0)$
along the curve $\gamma(t)$. We also write
$V=V(t,\bar{t})=\frac{\partial\Gamma}{\partial t}(t,\bar{t})$,
$W=W(t,\bar{t})$, $v=v(t,\bar{t})=\vert V(t,\bar{t})\vert$,
$T=T(t,\bar{t})$, $N=N(t,\bar{t})$, $B=B(t,\bar{t})$, etc., with
the obvious meanings and put $V(s,\bar{t})$, $W(s,\bar{t})$ etc.,
for the corresponding reparametrizations by arc-length. Then, the following general formulas for the variations of $v$ and $\kappa$ in $\gamma$, in the direction of the variation vector field $W$ can be obtained using
standard computations that involve the Frenet equations
\eqref{frenet1}-\eqref{frenet3} (see, \cite{Arroyo-Garay-Pampano}, \cite{Langer-Singer} and references therein)
\begin{eqnarray}
W(v) & = & v\langle\nabla_T W,T\rangle, \label{v1}\\
W(\kappa) & = & \langle\nabla^2_T W,N\rangle-2\kappa\langle\nabla_T W,T\rangle +\rho\langle W,N\rangle. \label{v2} 
\end{eqnarray}
Next, after a standard
computation involving integration by parts and formulae
\eqref{v1} and \eqref{v2}, the \emph{First Variation Formula} is obtained:
\begin{equation}
\frac{d}{d\nu}\mathbf{\Theta}(\nu)_{\mid _{\nu=o}}=\
 \int_{0}^{L}\langle\mathcal{E}(\gamma),W\rangle ds +\mathcal{B}
 \left[ W,\gamma \right] _{0}^{L}.  \label{1fv} \nonumber
\end{equation}
Here, $\mathcal{E}(\gamma)$ and  $\mathcal{B}\left[ W,\gamma \right]
_{0}^{L}$ denote the \emph{Euler-Lagrange operator} and \emph{boundary term},
respectively. These are given by
\begin{eqnarray}\label{elo}
&&\mathcal{E}(\gamma)  =\nabla_T \mathcal{J}-R(\mathcal{K},T)T=\nabla_T \mathcal{J}+\rho\,\mathcal{K}\, ,\quad\nonumber  \\ 
&&\mathcal{B}\left[ W,\gamma \right] _{0}^{L} = \left[\langle\mathcal{K},\nabla_T W\rangle - \langle \mathcal{J},W\rangle\right] _{0}^{L},\nonumber
\end{eqnarray}
where
\begin{eqnarray}
\mathcal{K}(\gamma) &=&\frac{1}{4\kappa^{3/4}}\, N\,,  \label{kfield} \\
\mathcal{J}(\gamma)&
=&\nabla_T \mathcal{K}-\frac{1}{2}\kappa^{1/4}\,T\,
. \label{ejk}
\end{eqnarray}

We will call \emph{critical curve} or \emph{extremal curve} to any curve $\gamma\subset\Omega^{\rho}_{p_op_1}$ such that $\mathcal{E}(\gamma)=0$. Notice that this is an abuse of notation, since proper criticality depends on the boundary conditions, as it is clear from the First Variation Formula. However, under suitable boundary conditions, curves verifying $\mathcal{E}(\gamma)=0$ are going to be proper critical curves. Therefore, since for our purposes we just need to consider curves satisfying $\mathcal{E}(\gamma)=0$, for the sake of simplicity, from now on, we are going to use the name critical curve (or, extremal curve) to denote any curve $\gamma\subset\Omega_{p_o p_1}^\rho$ verifying  $\mathcal{E}(\gamma)=0$. 

Now, using the Frenet equations \eqref{frenet1}-\eqref{frenet3}, we can
see that $\mathcal{E}(\gamma)$ has no component in $T$, nor in $B$, while its
normal component can be expressed in terms of the
curvature of $\gamma$. Thus, after long straightforward
computations, $\mathcal{E}(\gamma)=0$ reduces to
\begin{eqnarray}
\kappa^{3/4}\frac{d^2}{ds^2}\left(\frac{1}{\kappa^{3/4}}\right)-3\kappa^2+\rho&=&0\, , \label{Euler-Lagrange}
\end{eqnarray}
which is the \emph{Euler-Lagrange equation} for the curvature
energy functional $\mathbf{\Theta}$, \eqref{curveture-energy}, acting on $\Omega^{\rho}_{p_op_1}$ and agrees with formula (30) of \cite{Caddeo-Montaldo-Oniciuc-Piu}.

Non-geodesic critical curves with constant curvature are given by the only planar curves (up to isometries) whose curvature verifies
\begin{equation}
\kappa^2=\kappa_o^2=\frac{\rho}{3}\, , \label{constantcurvature}
\end{equation}
which is only possible in the case of the round 2-sphere, that is, if $N^2(\rho)=\mathbb{S}^2(\rho)$.

On the other hand, for critical curves with non-constant curvature, let us now define the following vector field $\mathcal{I}$ along
$\gamma$
\begin{equation}
\mathcal{I}=T\times\mathcal{K}\, , \nonumber
\end{equation}
where $\times$ denotes the cross product and $\mathcal{K}$ is
defined in \eqref{kfield}. Combining the Frenet equations
\eqref{frenet1}-\eqref{frenet3} and \eqref{kfield}, we see that $\mathcal{I}$ is
given by
\begin{equation}
\mathcal{I}=T\times\mathcal{K}=\frac{1}{4\kappa^{3/4}}\,
B\, .\label{I}
\end{equation}
Then, a direct long computation using the Frenet equations \eqref{frenet1}-\eqref{frenet3}, formulas \eqref{kfield} and \eqref{ejk}, and the Euler-Lagrange
equation, \eqref{Euler-Lagrange},
shows that the derivative of the function $\langle
\mathcal{J},\mathcal{J}\rangle+\rho\,
\langle\mathcal{I},\mathcal{I}\rangle$ along the critical curve
is zero. Thus,
\begin{eqnarray}
\langle\mathcal{J},\mathcal{J}\rangle+\rho\, \langle\mathcal{I},\mathcal{I}\rangle&=&d,\label{integral}
\end{eqnarray}
with $d$ a real constant, represents a first integral of \eqref{Euler-Lagrange}. Notice that substituting the values of $\mathcal{J}$, \eqref{ejk}, and $\mathcal{I}$, \eqref{I}, in above formula we get
\begin{equation}
\kappa_s^2=\frac{16}{9}\kappa^2\left(16\, d \,\kappa^{3/2}-9\,\kappa^2-\rho\right)\, .\label{firstintegral}
\end{equation}

We point out that, in principle, the constant $d$ may be arbitrary. However, as we will see later, for our purposes it will be restricted to be positive.

Finally, to end this section, we are going to see that critical curves for $\mathbf{\Theta}$, \eqref{curveture-energy}, have a distinguished vector field along them. A vector field $W$ along $\gamma$, which infinitesimally
preserves unit speed parametrization, is said to be a \emph{Killing
vector field along $\gamma$} (in the sense of \cite{Langer-Singer})
if $\gamma$ evolves in the direction of $W$ without changing
shape, only position. In other words,  the following equations
must hold
\begin{equation}
W(v)(s,0)=W(\kappa)(s,0)=0\, , \nonumber
\end{equation}
($v=\lvert \gamma'\rvert=\lvert \frac{d\gamma}{ds}\rvert$ being
the speed of $\gamma$) for any variation of $\gamma$ having $W$ as
variation field.

It turns out that extremals of
$\mathbf{\Theta}$, \eqref{curveture-energy}, have a naturally associated Killing vector field defined along
them, as we summarise in the following proposition (for a proof, see \cite{Arroyo-Garay-Pampano} and \cite{Langer-Singer})

\begin{prop}\label{carcenter}
Assume that $\gamma$ is an immersed curve in $N^2(\rho)$ which is an extremal of
$\mathbf{\Theta}$, \eqref{curveture-energy}.
Consider the vector field \eqref{I}
\begin{equation}
\mathcal{I}=\frac{1}{4\kappa^{3/4}}\,B \, ,\nonumber
\end{equation}
defined on $\gamma$, $B$ being its Frenet binormal vector field. Then  $\mathcal{I}$ is a Killing vector field along $\gamma$.
\end{prop}

\section{Characterisation of Profile Curves as Bending-Type Energy Extremals}

Throughout this section we are going to assume that $S$ is a non-CMC biconservative surface of a Riemannian 3-space form, $N^3(\rho)$. Then, as mentioned in \S 2, $S$ is locally rotational. We denote by $\gamma$ the curve everywhere orthogonal to the rotation, then $S$ can be locally parametrized as
\begin{equation}
x(s,t)=\phi_t(\gamma(s)), \label{par}
\end{equation}
where $\phi_t$ denotes the one-parameter group of rotations. Usually, $\gamma$ is called the \emph{profile curve} of $S_\gamma\subset S$. 

Profile curves of rotational surfaces are planar and, furthermore in this case, they have a nice geometric property, as stated in Theorem \ref{variationalcharacterization}. We prove this theorem in the following subsection.

\subsection{Proof of Theorem \ref{variationalcharacterization}}
Let $ S \subset N^3(\rho)$ be an isometrically
immersed  non-CMC biconservative surface in any Riemannian $3$-space form
$N^3(\rho)$ with local orientation determined by the normal
vector $\eta$. Then, by Proposition \ref{LW}, it is a rotational surface verifying the relation \eqref{LWrelationbiconservative} between its principal curvatures. We will denote by $\xi$ the Killing vector field which is the infinitesimal generator of the rotation that leaves $S$ invariant. Then, locally on
$S$, we can choose Fermi geodesic coordinates $(U,x)$,
$x:U\rightarrow S,\, x(s,t)$, so that
$\xi=\frac{\partial}{\partial t}$ and $s$ measures the arc-length
along geodesics orthogonal to $\xi$. Thus, calling
$\gamma(s):=x(s,0)$, we have  that $x(U):=S_\gamma\subset S$ is
parametrized  by \eqref{par} where $\phi_t\in \mathcal{G}_\xi$, the one-parameter group of isometries generated by $\xi$. Observe that $\gamma(s)$ and
all its copies by the action of $\mathcal{G}_\xi$,
$\gamma_t(s):=\phi_t(\gamma(s)), t\in \mathbb{R}$,  are arc-length
parametrized geodesics of $S_\gamma$ which are orthogonal to
$\xi$, so that $S_\gamma$  is foliated by geodesics having
$\kappa(s,t)$ as curvature in
$N^3(\rho)$. Furthermore, they all have vanishing torsion. If $\gamma_t(s)$ were also geodesics in
$N^3(\rho)$, $\forall t$, then $S_\gamma$ would be foliated by
geodesics of the ambient space what would make it a ruled surface. In this case, we have that from relation \eqref{LWrelationbiconservative}, $S_\gamma$ is minimal, since $\kappa_1=-\kappa(s)=0$, which is not possible.
Hence, we assume that the orthogonal curves to the Killing field
$\xi$, $\gamma_t(s)$, are not geodesics of the ambient space.
Then, $\gamma_t(s)$ are Frenet curves and defined over them we
have a Frenet frame $\{T(s,t),N(s,t),B(s,t)\}$ satisfying \eqref{frenet1}-\eqref{frenet3}.\\
At this point, after long straightforward computations, one can see  that the
Gauss and Weingarten formulae and the
simplicity of the curvature tensor in $N^3(\rho)$, lead  to a PDE system to be satisfied (see, for instance, \cite{Arroyo-Garay-Pampano}). The compatibility conditions for this
system are given by  the Gauss-Codazzi equations, which in our case, since $\phi_t$ are isometries,
can be shown to boil down just to
\begin{eqnarray}
\,0 & =&
\left(\frac{1}{\kappa}\left(G_{ss}+G(\kappa^2+\rho)\right)\right)_{s}-\kappa_s G\, , \label{Gauss-Codazzi}
\end{eqnarray}
where G is the length of the Killing vector field $\xi$, that is, $G^2(s)=\langle x_t,x_t\rangle$. Moreover, not only $G(s)$, but also all the involved functions depend only on $s$. Now, $\kappa_1(s)=-\kappa(s)$ and $\kappa_2(s)=h_{22}(s)$, the second coefficient of the second fundamental form given by (for details, see \cite{Arroyo-Garay-Pampano})
\begin{equation}
h_{22}=\frac{1}{\kappa}\left(\frac{G_{ss}}{G}+\rho\right), \nonumber
\end{equation}
and, therefore, the relation \eqref{LWrelationbiconservative} becomes
\begin{eqnarray}
G_{ss}=G\left(3\kappa^2-\rho\right).\label{three}
\end{eqnarray}
Let us assume first that $\gamma$ has constant curvature
$\kappa(s)=\kappa_o$ in $N^3(\rho)$. We combine (\ref{Gauss-Codazzi}) and
(\ref{three}) to obtain that $G(s)$ must be a positive constant and, therefore, $S_\gamma$ should be a flat isoparametric surface which contradicts the fact that $S$ has non-CMC. Consequently, it is out of our consideration. Even though, in this case, equation \eqref{three} implies that 
\begin{equation}
3\kappa_o^2=\rho\, , \nonumber
\end{equation}
that is, $\gamma$ is also a critical curve with constant curvature for $\mathbf{\Theta}$, \eqref{curveture-energy}, (see formula \eqref{constantcurvature}).

\noindent Finally, suppose that $\kappa$ is not constant. Locally, by the
Inverse Function Theorem we can suppose that $s$ is a function of
$\kappa$ and calling $G(\kappa)=\dot{P}(\kappa)$, where the
upper dot denotes derivative with respect to $\kappa$, we have
that \eqref{Gauss-Codazzi} and \eqref{three} can be expressed
in the following way
\begin{eqnarray}
\dot{P}_{ss}+\dot{P}\left(\kappa^2+\rho\right)-\kappa \left(P+\lambda\right)&=&0\, , \label{1}\\
\dot{P}_{ss}-\dot{P}\left(3\kappa^2-\rho\right)&=&0\, , \label{2}
\end{eqnarray}
for some $\lambda\in \mathbb{R}$. Now, equation \eqref{1} is the Euler-Lagrange equation for $\int_{\gamma} (P(\kappa)+\lambda)ds$ in $N^3(\rho)$ (see for instance, \cite{Arroyo-Garay-Pampano}). Moreover, substituting it in equation \eqref{2} we get an ODE in $P$ which can be solved obtaining 
\begin{equation}
P(\kappa)=\kappa^{1/4}-\lambda\, . \nonumber
\end{equation}
Thus, $\gamma$ must be a critical curve for $\mathbf{\Theta}$, \eqref{curveture-energy}, proving the result. \hfill $\square$
\\

In fact, as mentioned in the introduction, the converse of Theorem \ref{variationalcharacterization} is also true and gives us a way of constructing all non-CMC biconservative surfaces of 3-space forms after binormal evolution of extremal curves, as we will explain in what follows.

From Proposition \ref{carcenter}, we know that the vector field along $\gamma$, $\mathcal{I}$, \eqref{I}, is a Killing vector field along the curve. Therefore, using an argument similar to that of \cite{Langer-Singer} we can extend $\mathcal{I}$  to a Killing
vector field on the whole $N^3(\rho)$. Let us denote it by $\mathcal{I}$
again. Since $N^3(\rho)$ is complete, we can consider the
one-parameter group of isometries determined by the flow of
$\mathcal{I}$, $\{\phi_t\,;\,t\in \mathbb{R}\}$, and define the
surface $S_\gamma:=\{\phi_t(\gamma(s))\}$ obtained as the
evolution of $\gamma$ under the $\mathcal{I}$-flow. Observe that
$S_\gamma$ is an $\mathcal{I}$-invariant surface, which is foliated by congruent
copies of $\gamma$, $\gamma_t(s):=\phi_t(\gamma(s))$.

Moreover, since $\phi_t$ are isometries of $N^3(\rho)$, we have
\begin{equation}
x_t(s,t)= \frac{1}{4\kappa^{3/4}}\, B(s,t)\, ,\nonumber
\end{equation}
$\kappa(s)$ being the curvature of $\gamma(s)$, and $B(s,t)$ the
unit Frenet binormals of $\gamma_t(s)$. Thus, $S_\gamma$ obtained
as the flow evolution of $\gamma$, $x(s,t)=\phi_t(\gamma(s))$, is
a \emph{binormal evolution surface} with velocity
$V(s):=\langle x_t,x_t\rangle^{\frac{1}{2}}
=\langle\mathcal{I},\mathcal{I}\rangle^\frac{1}{2}$
(for more details see, \cite{Arroyo-Garay-Pampano}). 

Now, if $S_\gamma$ denotes a binormal evolution surface all whose filaments satisfy $\tau=0$, then, as proved in \cite{Arroyo-Garay-Pampano}, the fibers of $S_\gamma$ have constant curvature and zero torsion (if they are not geodesics) in $N^3(\rho)$. In particular, if the curvature of the filaments, $\kappa(s,t)$, is also constant, then $S_\gamma$ is a flat isoparametric surface. 

For the case where the filaments have non-constant curvature, the following proposition was proved in \cite{Arroyo-Garay-Pampano}.

\begin{prop}\label{rot} Let $S_\gamma\subset N^3(\rho)$ be a binormal evolution surface all whose filaments have zero torsion. Then, if they also have non-constant curvature, $S_\gamma$ is a rotational surface.
\end{prop}

Thus, using these facts together with equation \eqref{LWrelationbiconservative}, we can prove the converse of Theorem \ref{variationalcharacterization}.

\begin{thm}\label{converse} Let $\gamma$ be a planar extremal curve with non-constant curvature of the energy $\mathbf{\Theta}(\gamma)=\int_{\gamma}\kappa^{1/4}$ and let $S_\gamma$ denote the $\mathcal{I}$-invariant surface in $N^3(\rho)$ obtained by evolving $\gamma$ under the flow of the Killing field $\mathcal{I}$ which extends \eqref{I} to $N^3(\rho)$. Then, $S_\gamma$ is a rotational linear Weingarten surface of $N^3(\rho)$ verifying \eqref{LWrelationbiconservative}, that is, $S_\gamma$ is a non-CMC biconservative surface.
\end{thm}
\begin{proof} Take any planar extremal curve of $\mathbf{\Theta}$, \eqref{curveture-energy}, then as explained above, we can locally define the $\mathcal{I}$-invariant surface $S_\gamma=\{\phi_t(\gamma(s)\}$, where $\{\phi_t\,;\, t\in\mathbb{R}\}$ is the one-parameter group of isometries determined by $\mathcal{I}$. Furthermore, the square of the length of the Killing vector field $\mathcal{I}$ is given by
\begin{equation}
V^2(s)=\langle\mathcal{I},\mathcal{I}\rangle=\frac{1}{16\kappa^{3/2}}\, . \label{Gs}
\end{equation}
Then, as the evolution is made by isometries, $\gamma$ and all its congruent copies are planar extremals of $\mathbf{\Theta}$, \eqref{curveture-energy}. Now, from Proposition \ref{rot} we have that $S_\gamma$ is a rotational surface. Finally, any $\gamma_t$ verifies the Euler-Lagrange equation \eqref{Euler-Lagrange}, which is, using \eqref{Gs}, equivalent to
\begin{equation}
\frac{V_{ss}}{V}=3\kappa^2+\rho\, .\nonumber
\end{equation}
Thus, using that $\kappa_1=-\kappa$ and $\kappa_2=h_{22}$ we get $3\kappa_1+\kappa_2=0$. That is, $S_\gamma$ is a rotational linear Weingarten surface verifying \eqref{LWrelationbiconservative}. 
\end{proof}

Notice that Theorem \ref{converse} gives a way of constructing non-CMC biconservative surfaces of $N^3(\rho)$. In fact, together with Theorem \ref{variationalcharacterization}, it characterises non-CMC biconservative surfaces as the binormal evolution surfaces generated by a planar extremal of $\mathbf{\Theta}$, \eqref{curveture-energy}. This characterisation also allows us to analyse global properties of the binormal evolution surfaces based on topological facts about the profile curves. In \cite{Nistor}, \cite{Nistor-Oniciuc} and \cite{Nistor-Oniciuc-2}, the existence of complete non-CMC biconservative surfaces has been proved for both $\mathbb{R}^3$ and $\mathbb{S}^3(\rho)$. Moreover, in \cite{Barros-Garay}, the authors have proved the existence of complete non-compact rotational surfaces verifying the linear relation \eqref{LWrelationbiconservative} between their principal curvatures when $\rho\leq 0$. In \S 5, making use of our characterisation of the profile curve, we are going to study the existence of non-CMC closed biconservative surfaces.

\section{Closed Non-CMC Biconservative Surfaces of 3-Space Forms}

The main purpose of this section is to study the existence of closed (compact without boundary) non-CMC biconservative surfaces in 3-space forms. To fulfill this objective, we are going to use the characterisation introduced in the previous section. First of all, we need the orbits of the rotation to be closed, that is, euclidean circles. Notice that the value of the constant of integration $d$ plays an essential role, as proved in \cite{Arroyo-Garay-Pampano}. In fact, the orbits of the rotation are euclidean circles if and only if $d$ is positive. Therefore, we need to constraint the constant of integration and, after that, we have two options in order to obtain closed surfaces. On one hand, if the critical curve cuts the axis of rotation sufficiently many times, then the rotational surface will be closed. On the other hand, closed critical curves also give rise to closed surfaces. 

Observe that a critical curve $\gamma$ is completely determined (up to rigid motions) by its curvature, $\kappa(s)$, which must be a solution of the first integral of the Euler-Lagrange equation \eqref{firstintegral}. Now, we need the right hand side of equation \eqref{firstintegral} to be positive. For notation convenience we write $u=\kappa^{1/2}$ and, therefore, equation \eqref{firstintegral} reads
\begin{equation}
u_s^2=\frac{4}{9}u^2\left( 16\, d\, u^3-9\,u^4-\rho\right).\label{u}
\end{equation}
Then, the following polynomial must be positive for some values of $u$
\begin{equation}
Q(u)=16\, d\, u^3-9\, u^4-\rho>0\,. \label{condition}
\end{equation}
We have that $Q(u)$ tends to $-\infty$, whenever $u$ tends to either $+\infty$ or $-\infty$. Moreover, $u=4d/3$ represents a (local) maximum for $Q(u)$. Therefore, condition \eqref{condition} is verified for some values of $u$ if and only if $Q(4d/3)>0$, which gives an extra constraint on the parameter $d$.

\begin{figure}[H]
\begin{center}
\includegraphics[width=.4\textwidth]{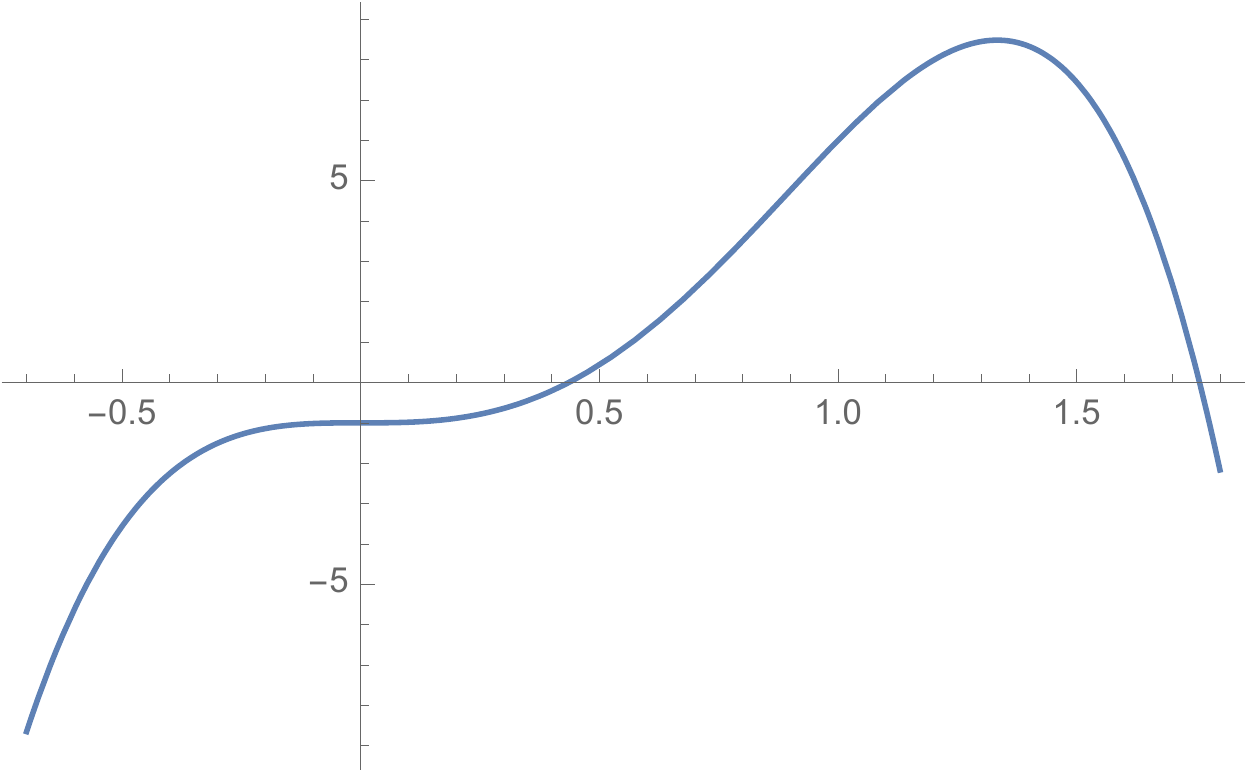}
\end{center}
\caption{Plot of the polynomial $Q(u)$ for $\rho=1$ and $d=1$.}\label{polynomial}
\end{figure}

To be more precise, this extra constraint only appears when $\rho>0$ (since for $\rho\leq 0$ is always true), and in this case we have
\begin{equation}
d>d_*=\frac{\left(27\,\rho\right)^\frac{1}{4}}{4}\, . \label{constraintd}
\end{equation}

Notice that this argument also shows the existence of just two roots of $Q(u)$ (see Figure \ref{polynomial}). Let us call $\alpha$ and $\beta$ these roots, where $\beta<\alpha$. Reversing the change of variable $u=\kappa^{1/2}$, they will become the maximum and minimum curvatures of the profile curve $\gamma$, respectively. Indeed, we have $\beta<u<\alpha$ for any $u$ that verifies \eqref{condition}.

If the profile curve $\gamma$ happens to cut the axis of rotation, then there will be some fixed points in the evolution under the $\mathcal{I}$-flow. However, from Proposition \ref{carcenter}, we have that the $\mathcal{I}$-flow has fixed points along $\gamma$ if and only if the curvature, $\kappa(s)$, tends to infinity, which is not possible since $u$ (and, therefore, the curvature) is bounded. Thus, the only option to find closed surfaces is that the profile curve is closed.

Observe that a necessary, not sufficient, condition for a curve to be closed is to have periodic curvature. Let us assume for a moment that there exist critical curves for $\mathbf{\Theta}$, \eqref{curveture-energy}, with periodic curvature, then we can obtain conditions for both $\gamma$ and $S_\gamma$ to be closed. Indeed, adapting the computations of \cite{Arroyo-Garay-Mencia}, if we define the function
\begin{equation}
\Lambda(d)=12\int_o^\varrho \frac{\kappa^{7/4}}{16\,d\,\kappa^{3/2}-\rho}\, ds\, , \label{Lambda(d)}
\end{equation}
where $\varrho$ is the period of $\kappa(s)$ and $d>0$ is the constant of integration given by \eqref{integral}, we have the following sufficient condition. 

\begin{prop}\label{closed} Let $\gamma\subset N^3(\rho)$ be a planar critical curve for $\mathbf{\Theta}$, \eqref{curveture-energy}, with periodic curvature $\kappa(s)$, then $\gamma(s)$ is closed if and only if the function $\Lambda(d)$, \eqref{Lambda(d)}, vanishes for $\rho\leq 0$, or it is equal to $\frac{2\,n\,\pi}{m\,\sqrt{\rho\, d}}$, for some integers $n$ and $m$, when $\rho>0$.
\end{prop}

Now, making use of Proposition \ref{closed}, the following result is clear, since, for $\rho\leq 0$, the integrand of \eqref{Lambda(d)} is always positive and, therefore, $\Lambda(d)$ never vanishes.

\begin{prop}\label{non-existence-closed} There are no closed non-CMC biconservative surfaces in 3-space forms, $N^3(\rho)$, with $\rho\leq 0$.
\end{prop}

We point out that Proposition~\ref{non-existence-closed} can be also deduced from \cite{Nistor-Oniciuc-3} by using a different approach. 

If $N^3(\rho)=\mathbb{S}^3(\rho)$, we will prove the existence of closed non-CMC biconservative surfaces. We begin by checking that there are critical curves of $\mathbf{\Theta}$, \eqref{curveture-energy}, in $\mathbb{S}^2(\rho)$ whose curvature is periodic. What is more, we have the following proposition.

\begin{prop}\label{morse} When defined in the whole real line, all critical curves for $\mathbf{\Theta}$, \eqref{curveture-energy}, in $\mathbb{S}^2(\rho)$ have periodic curvature.
\end{prop}
\begin{proof} Let $\gamma(s)$ be a critical curve for $\mathbf{\Theta}$, \eqref{curveture-energy}. Then, the non-constant curvature of $\gamma(s)$ must be a solution of the first integral of the Euler-Lagrange equation \eqref{firstintegral}, where $d>d_{*}$, see \eqref{constraintd}. To simplify notations we put $x=u=\kappa^{1/2}$ and $y=x_s$. Then, \eqref{firstintegral} can be rewritten as (see \eqref{u})
$$y^2=\frac{4}{9}x^2\left(16 d x^3-9x^4-\rho\right)=\frac{4}{9}x^2 Q(x).$$
This is an algebraic curve which, by the standard square root method of algebraic geometry and above analysis of the polynomial $Q(x)$ (see Figure \ref{polynomial}), it is closed for any $d>d_{*}$. 
Thus, the curve $c(s)=(x(s), y(s))$ is included in the trace of the compact regular curve $y^2=({4}/{9})x^2 Q(x)$ and it can be thought as a bounded integral curve of the smooth vector field
$$
X(\kappa,z)=\left(z,\frac{1}{\sqrt{\kappa}}\left[\frac{5}{2}z^2+\frac{2}{3}\rho \kappa-2\kappa^3\right]\right)
$$
defined in $\{(\kappa,z)\in \r^2 \colon \kappa>0 \}$.
This implies that $c(s)$ is smooth and defined on the whole $\r$. Finally, since the vector field $X(\kappa,z)$ has no zeros along the curve $c(s)$ when $d>d^*$, we conclude, applying the Poincare-Bendixon Theorem, that $c(s)$ is a periodic curve.
\end{proof}

\begin{rem}\label{rem-simple}
Of course, since the profile curve has periodic curvature, the binormal evolution surface generated by it is complete. Moreover, using the differential equation \eqref{u} satisfied by $u=\kappa^{1/2}$, it is easy to check that when $u=\alpha$ or $u=\beta$ the vector field $\mathcal{J}$ has only component in $T$, that is, the profile curve $\gamma$ is parallel to the integral curves of the Killing vector field $\mathcal{J}$ in that points. Therefore, our curve is bounded between those parallels. What is more, in those points the length of $\mathcal{J}$ is never zero, since, both $\alpha$ and $\beta$ are positive. This means that $\gamma$ does not cross over the pole of the parametrization.  In fact, since the component in $T$ of the Killing vector field $\mathcal{J}$ is a non-zero multiple of $u^{1/2}$ and $u$ is always positive (it varies from $\alpha$ to $\beta$, which are, in the spherical case, positive since $Q(0)<0$), we get that $\gamma$ is never orthogonal to the integral curves of $\mathcal{J}$, that is, $\gamma$ is always going forward. Consequently, it does not cut itself in one period of its curvature, unless it gives more than one round in that period.
\end{rem}

Now, in order to assure closure, we have seen that a binormal evolution surface of $\mathbb{S}^3(\rho)$ whose profile curve $\gamma$ has periodic curvature, $\kappa(s)$, and vanishing torsion is a closed surface if and only if the function $\Lambda(d)$, \eqref{Lambda(d)}, verifies
\begin{equation}
\Lambda(d)= 12\int_o^\varrho \frac{\kappa^{7/4}}{16\,d\,\kappa^{3/2}-\rho}\, ds =\frac{2\,n\,\pi}{m\,\sqrt{\rho\, d}}\,,\label{relation-condition}
\end{equation}
for some $d>0$ and some integers $m$ and $n$ with ${\rm gcd}(m,n)=1$. The integer $n$ represents the number of rounds the curve gives around the pole in order to close up, while $m$ is the number of lobes the curve has, that is, the number of periods of the curvature. In particular, a closed curve $\gamma$ is simple if and only if it closes up in one round, that is, if it verifies the closure condition for $n=1$.

To check the closure condition \eqref{relation-condition}, we need to study the image of the function $I(d)=\sqrt{\rho\, d}\, \Lambda(d)$ as $d$ varies in the domain \eqref{constraintd}. For this purpose, first we are going to state the following technical lemma (for the proof see \S 6).

\begin{lem}\label{lemma} The function $I(d)=\sqrt{\rho\, d}\, \Lambda(d)$ is strictly decreasing in $d$. Furthermore, for any $d\in\left( d_*,+\infty\right)$, it is bounded by
\begin{equation}
\pi<I(d)=\sqrt{\rho\, d}\, \Lambda(d)<\sqrt{2}\pi\, . \nonumber
\end{equation}
\end{lem} 

Summarising our findings we obtain the proof of Theorem \ref{complete-case} as mentioned in the introduction.

\subsection{Proof of Theorem~\ref{complete-case}}

Let $m$ and $n$ be two integers such that ${\rm gcd}(m,n)=1$ and $m<2\,n<\sqrt{2}\, m$. Then 
$$
\pi<\frac{2\,n\,\pi}{m}<\sqrt{2}\,\pi\,.
$$
Now, from Lemma \ref{lemma}, the function $I(d)=\sqrt{\rho\, d}\, \Lambda(d)$ varies from $\pi$ to $\sqrt{2}\pi$ as $d$ decreases from $+\infty$ to $d_*$. Thus, there exists a  $d_{m,n}> d_*$, such that the relation \eqref{relation-condition} is verified and, therefore, the corresponding associated non-CMC biconservative surface is closed. 

\noindent Furthermore, the corresponding  surface is embedded if the profile curve is simple. Now, using Remark~\ref{rem-simple}, the profile curve is simple if and only if it closes up in one round. Therefore, when $n=1$, we need that the closure condition is  satisfied for some integer $m$. That is, we need the existence of an integer $m$ such that 
\begin{equation}
\pi<\frac{2\,\pi}{m}<\sqrt{2}\pi\, .\nonumber
\end{equation}
However, the above relation is not possible and, therefore, there are not closed non-CMC biconservative surfaces embedded in $\mathbb{S}^3(\rho)$, as stated. \hfill $\square$\\

From the proof of Theorem \ref{complete-case} we deduce that there exists a discrete biparametric family of closed non-CMC biconservative surfaces in $\mathbb{S}^3(\rho)$. In fact, we have a closed non-CMC biconservative surface for any couple of integers $m$ and $n$ such that $m<2\,n<\sqrt{2}\, m$. The first one corresponds to $n=2$ and $m=3$, that is, the binormal evolution surface with initial condition a critical curve for $\mathbf{\Theta}$, \eqref{curveture-energy}, which has 3 lobes and needs 2 rounds around the pole to close up. We explain this in Figure \ref{ProfileCurve}. The green part of  the curve corresponds with the part of the critical curve covered in one period of the curvature. Notice that, as the curvature is the same for each period of it, our critical curve is nothing but congruent copies of the green part, that is, the whole curve can be constructed by gluing smoothly $m$ copies (in these particular cases $m=3$ and $m=5$ copies, respectively) of the trace covered in one period of the curvature.

\begin{figure}[H]
\begin{center}
\hspace{-0.35cm}\includegraphics[width=5cm,height=5cm]{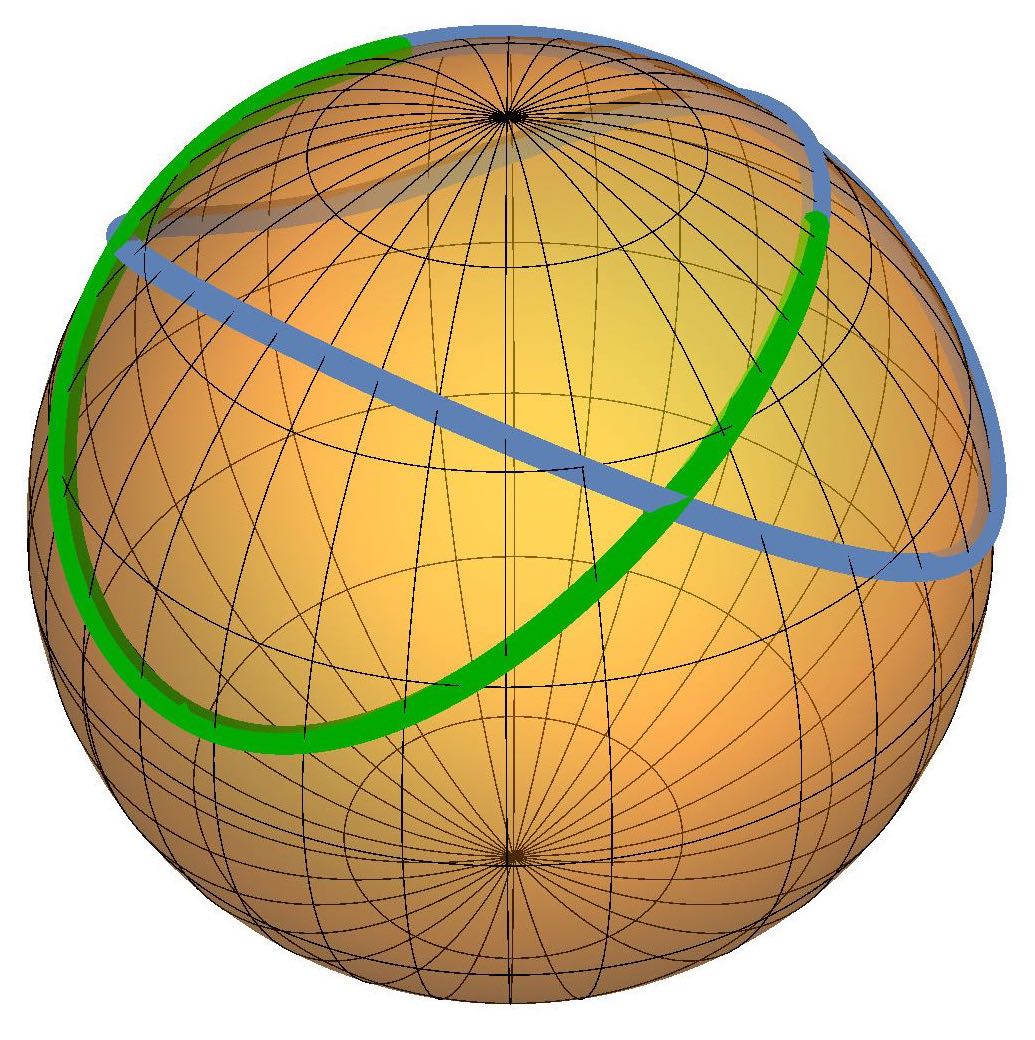}\quad\quad\quad\includegraphics[width=5cm,height=5cm]{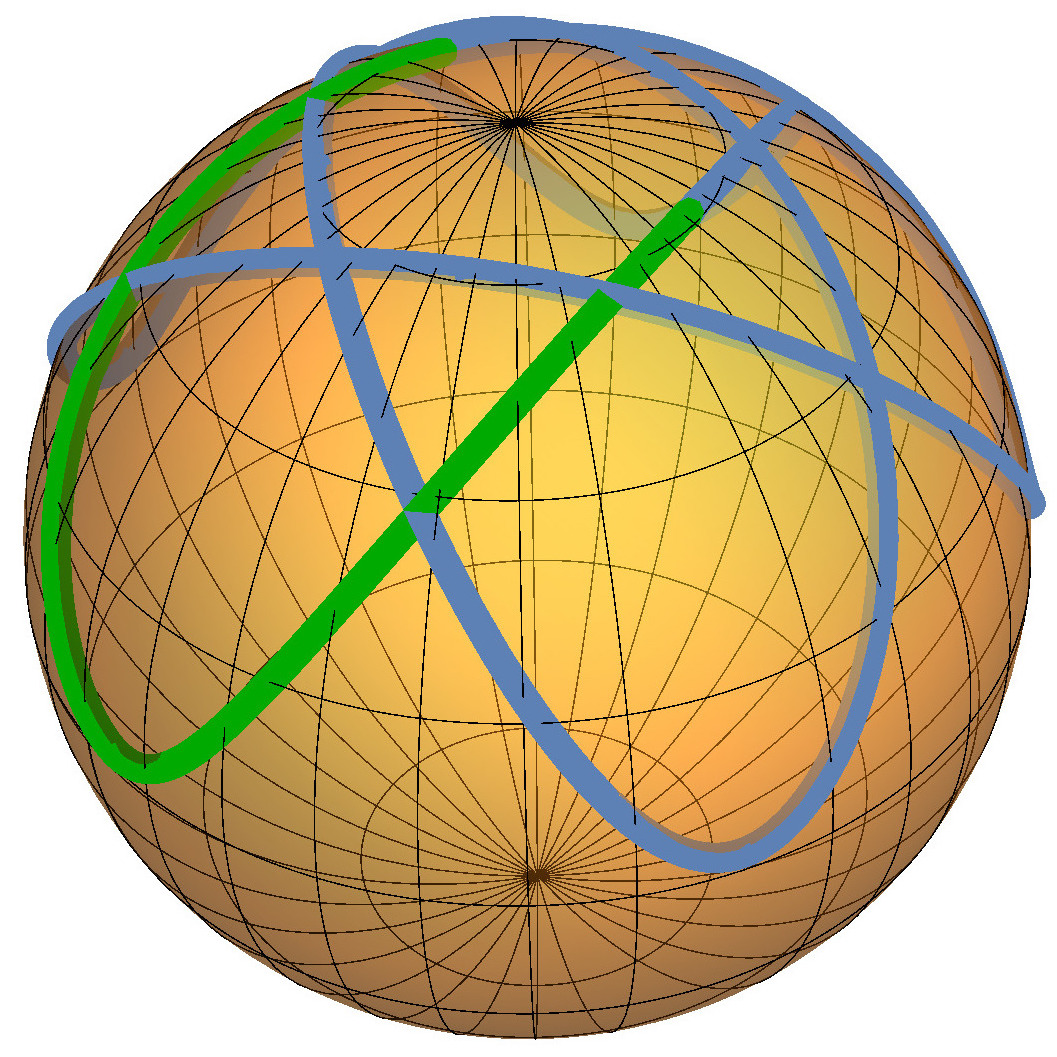}
\end{center}
\vspace{-0.4cm}
\caption{Closed critical curves for $\mathbf{\Theta}$, \eqref{curveture-energy}, in $\mathbb{S}^2(\rho)$ for $m=3$ and $n=2$ (Left) and $m=5$ and $n=3$ (Right).}\label{ProfileCurve}
\end{figure}

Furthermore, as proved in previous sections, all non-CMC biconservative surfaces are binormal evolution surfaces with initial condition a planar critical curve for $\mathbf{\Theta}$, \eqref{curveture-energy}. In the round 3-sphere, these binormal evolution surfaces $S_\gamma\subset\mathbb{S}^3(\rho)$ can be parametrized, up to an isometry of the ambient sphere, as:
\begin{equation}
x(s,\phi)=\frac{1}{4\sqrt{\rho\,d\,} \kappa^\frac{3}{4}}\left(\sqrt{\rho}\cos\phi,\sqrt{\rho}\sin\phi,\sqrt{16\,d\,\kappa^\frac{3}{2}-\rho} \sin\psi(s),\sqrt{16\,d\,\kappa^\frac{3}{2}-\rho}\cos\psi(s)\right), \nonumber
\end{equation}
where $\kappa(s)$ represents the curvature of $\gamma$, which is a solution of the Euler-Lagrange equation \eqref{Euler-Lagrange}, and $\psi(s)$ is given by
\begin{equation}
\psi(s)=-12\,\sqrt{\rho\, d} \int \frac{\kappa^{7/4}}{16\,d\,\kappa^{3/2}-\rho}\, ds\, .\nonumber
\end{equation}
Finally, notice that $\gamma(s)=x(s,0)$ is a parametrization of the profile curve. In Figure~\ref{Surface}, by using the above parametrization we show a plot of the stereographic projection of the closed non-CMC biconservative surface in $\mathbb{S}^3(\rho)$ for $m=3$ and $n=2$.

\begin{figure}[H]
\begin{center}
\includegraphics[width=6cm,height=6cm]{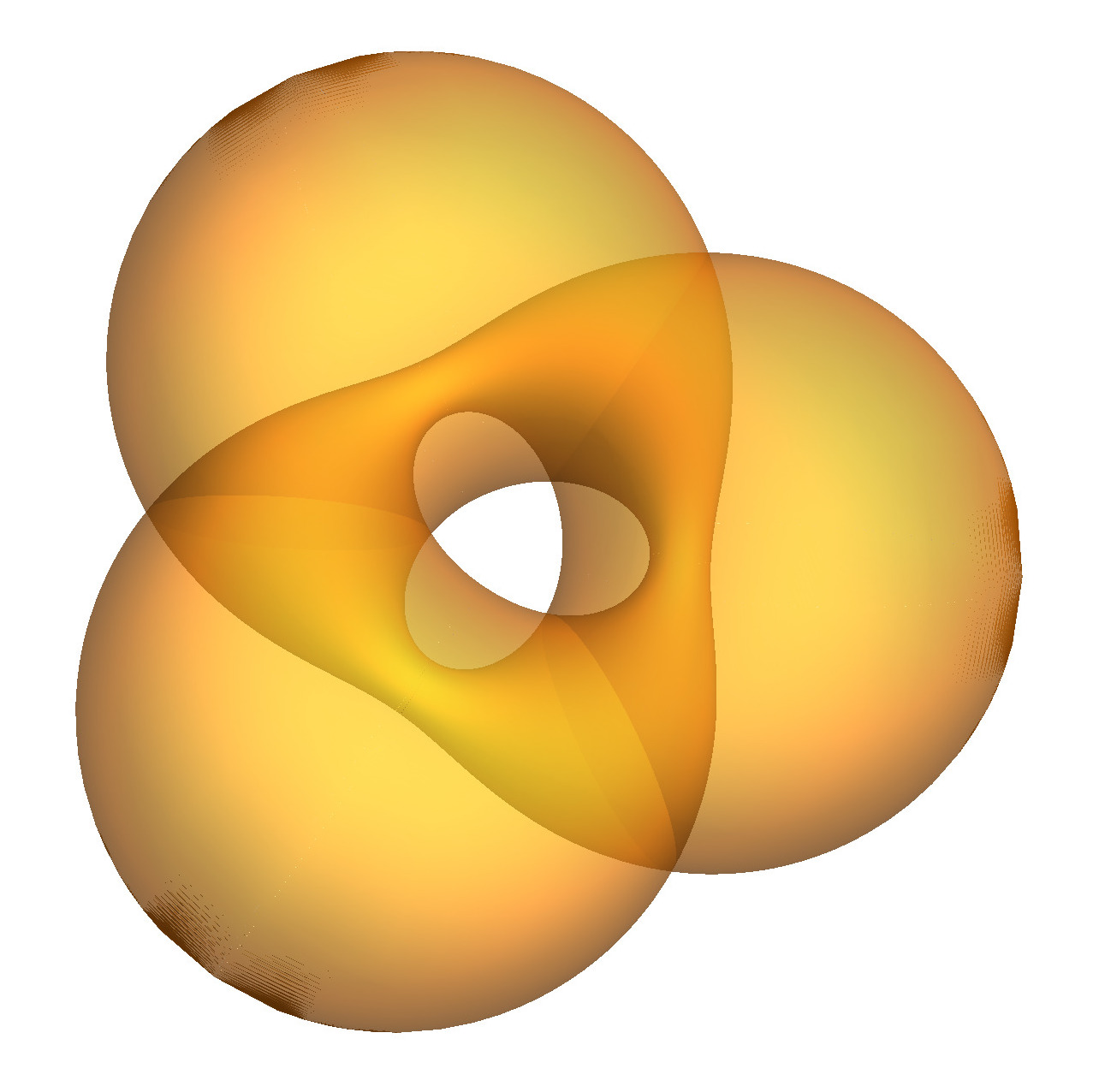}
\end{center}
\vspace{-0.5cm}
\caption{Stereographic projection of the closed non-CMC biconservative surface in $\mathbb{S}^3(\rho)$ for $m=3$ and $n=2$.}\label{Surface}
\end{figure}

\begin{rem}
Note that, from \eqref{LWrelationbiconservative}, we deduce immediately that the mean curvature $H$ of $S_\gamma$, along the profile curve $\gamma$, coincides with the curvature $\kappa$. Moreover, the parametrization  $x(s,\phi)$ is obtained by the action on  $\gamma(s)=x(s,0)$ of the one-parameter group of isometries generated by the Killing vector field of  $\mathbb{S}^3(\rho)$ given by $\mathcal{I}=\sqrt{d}(y\partial/\partial x-x\partial/\partial y)$. Then a point $\gamma(s_o)=x(s_o,0)$ generates a geodesic orbit of $S_\gamma$, under the action of $\mathcal{I}$, if $s_o$ satisfies $V_s (s_o)=0$ where $V(s)$ is the norm of $\mathcal{I}$ along $\gamma(s)$. A straightforward computation gives $V_s (s)=-3\kappa_s/(16\kappa^{7/4})$. We conclude that the points where the derivative of the curvature of the profile curve $\gamma$ vanishes determine closed geodesics on the surface and $\grad H$ vanishes along those geodesics.
\end{rem}
\section{Proof of Lemma \ref{lemma}}

Throughout this section we will prove Lemma \ref{lemma}. For this end, observe that, with the notation introduced in previous section and taking into account the symmetry, $\Lambda(d)$ can be written as
\begin{equation}
\Lambda(d)=24\int_o^{\varrho/2} \frac{u^{7/2}}{16\,d\,u^3-\rho}\, ds\, ,\nonumber
\end{equation}
Now, since in a half period of the curvature the function $u$ is increasing, we can use equation \eqref{u} to make a change of variable, obtaining that
\begin{equation}
\Lambda(d)=36\int_\beta^\alpha \frac{u^{5/2}}{\left(16\,d\,u^3-\rho\right)\sqrt{16\,d\, u^3-9\,u^4-\rho}}\, du\, . \label{expression}
\end{equation}
At this point, we divide our proof in three different parts.

\subsection{Part (i)} 

We begin by considering the limit of $I(d)=\sqrt{\rho\,d}\,\Lambda(d)$ when $d$ tends to $d_*$ (see the definition in \eqref{constraintd}). We will compute this limit with the aid of the Dirac's delta, $\delta(u-\varpi)$, since the limit of the integrand is zero everywhere but at $u=4d_*/3$, where it goes to infinity. This suggests that the integrand is a multiple of $\delta(u-\varpi)$ with $\varpi=4d_*/3$. Therefore, we first recall that any general Dirac's delta $\delta(u-\varpi)$ can be represented by the limit
\begin{equation}
\delta(u-\varpi)=\lim_{\varepsilon\rightarrow 0} \frac{\varepsilon}{\pi\left( (u-\varpi)^2+\varepsilon^2\right)}\, . \nonumber
\end{equation}
Let us multiply and divide the integrand of $\Lambda(d)$, \eqref{expression}, by this limit for $\varepsilon=d-d_*$. That is,
\begin{eqnarray}
\lim_{d\rightarrow d_*} I(d)&=&\lim_{d\rightarrow d_*} \sqrt{\rho\, d}\, \Lambda(d)=36\lim_{d\rightarrow d_*}\int_\beta^\alpha \frac{\sqrt{\rho\, d} \,u^\frac{5}{2}}{\left(16\,d\,u^3-\rho\right)\sqrt{16\,d\, u^3-9\,u^4-\rho}}\, du\nonumber\\ 
&=& 36 \int_\mathbb{R}\lim_{d\rightarrow d_*} \frac{\pi\sqrt{\rho\, d}\, u^\frac{5}{2}\chi_{(\beta,\alpha )}(u)\left(\left(u-\left(\frac{\rho}{3}\right)^\frac{1}{4}\right)^2+\left(d-d_*\right)^2\right)}{\left(d-d_*\right)\left(16\,d\,u^3-\rho\right)\sqrt{16\, d\, u^3-9\,u^4-\rho}}\,\delta\left(u-\left(\frac{\rho}{3}\right)^\frac{1}{4}\right)\, du\nonumber
\end{eqnarray}
Moreover, we recall that a nice property of these distributions that will be essential in this first part of the proof is the following
\begin{equation}
\int_\mathbb{R} f(u)\,\delta(u-\varpi)\, du=f(\varpi)\, ,\nonumber
\end{equation}
for any function $f$ and any constant $\varpi$. Notice that since we have taken $\varepsilon=d-d_*$, the limit $\varepsilon\rightarrow 0$ changes to $d\rightarrow d_*$. Indeed, this is quite convenience, since whenever $d$ is close to $d_*$, $\alpha$ and $\beta$ converge linearly in $d$ to $\left(\rho/3\right)^{1/4}$, that is, precisely to $4d_*/3$. Thus, by using this property we have
\begin{equation}
\lim_{d\rightarrow d_*} I(d)=   36 \lim_{d\rightarrow d_*} \frac{\pi\sqrt{\rho\, d}\, \left(\frac{\rho}{3}\right)^\frac{5}{8}}{\left(16\,d\,\left(\frac{\rho}{3}\right)^\frac{3}{4}-\rho\right)\sqrt{9\, \left(\frac{\rho}{3}\right)^\frac{1}{2}+\sqrt{3\,\rho}\left(1+2\left(\frac{\rho}{3}\right)^\frac{1}{4}\right)}} \,=\, \, \sqrt{2}\, \pi\, . \nonumber
\end{equation}

This limit can also be computed using Lemma 4.1 of \cite{Perdomo}. For our particular case, when $d=d_*$, the only root of the polynomial $Q(u)$, \eqref{condition}, is $u=4d_*/3$. Furthermore, as explained before, $4d_*/3$ is a local maximum of $Q(u)$ for $d=d_*$. Therefore, the result of this lemma can be summarised as follows
\begin{equation}
\lim_{d\rightarrow d_*} I(d)=\frac{36\sqrt{\rho\,d_*}\left(\frac{\rho}{3}\right)^{5/8}\pi}{\left(16\,d_*\,\left(\frac{\rho}{3}\right)^{3/4}-\rho\right)\sqrt{-\frac{1}{2}Q''\left(\left(\frac{\rho}{3}\right)^{1/4}\right)}}=\sqrt{2}\,\pi\, , \nonumber
\end{equation}
obtaining the desired result.

\subsection{Part (ii)}

Let us now consider the limit of $I(d)=\sqrt{\rho\,d}\,\Lambda(d)$ when $d$ goes to infinity. For this purpose, we need to work in the complex plane $\mathbb{C}$. We begin by defining the complex function
\begin{equation}
h(z)=-i\left(\sqrt{-z}\right)^5\left(\alpha-z\right)\sqrt{\frac{z-\beta}{\alpha-z}}\sqrt{\left(z-\omega_1\right)\left(z-\bar{\omega}_1\right)}\, \nonumber
\end{equation}
where $\omega_1$ and $\bar{\omega}_1$ are the two pure complex roots of $Q(u)$, and the square root symbol denotes the principal branch of it. 

If $U_1=\{ x+i\,y\in\mathbb{C}\,;\, y=0,\, \beta< x <\alpha\}$, then the Moebius transformation $\frac{z-\beta}{\alpha-z}$ maps $U_1$ to the set of positive real numbers $\mathbb{R}^+$. Now, $U_2=\{x+i\,y\in\mathbb{C}\, ;\, y=0,\, x<0\}$ and we obtain that $\sqrt{-z}$ is well-defined and analytic in $\mathbb{C}- U_2$. Finally, $\sqrt{\left(z-\omega_1\right)\left(z-\bar{\omega}_1\right)}$ is also analytic far from $\omega_1$ and $\bar{\omega}_1$. That is, the complex function
\begin{equation}
f(z)=\frac{36\,\sqrt{\rho\,d}\,z^5}{\left(16\,d\,z^3-\rho\right)h(z)}\, , \label{f(z)}
\end{equation}
is well-defined and holomorphic for any $z\in\mathbb{C}- \left(U_1\cup U_2\cup \{0,\beta,\alpha,\omega_1,\bar{\omega}_1,\omega_2,\bar{\omega}_2\}\right)$ where $\omega_2$ and $\bar{\omega}_2$ represent the pure complex roots of $16d z^3-\rho$. 

Moreover, notice that we have the following limits
\begin{eqnarray}
&& \,\lim_{\epsilon\rightarrow 0} \sqrt{-\left(x+i\,\epsilon\right)}\sqrt{\left(x+i\,\epsilon-\omega_1\right)\left(x+i\,\epsilon-\bar{\omega}_1\right)}=i\sqrt{x}\sqrt{\left(x-\omega_1\right)\left(x-\bar{\omega}_1\right)}\, , \nonumber\\
&& \lim_{\epsilon\rightarrow 0^+}\sqrt{\frac{x+i\,\epsilon-\beta}{\alpha-x-i\,\epsilon}}=\sqrt{\frac{x-\beta}{\alpha-x}}\, , \nonumber\\
&& \lim_{\epsilon\rightarrow 0^-}\sqrt{\frac{x+i\,\epsilon-\beta}{\alpha-x-i\,\epsilon}}=-\sqrt{\frac{x-\beta}{\alpha-x}}\, . \nonumber
\end{eqnarray}
And, therefore,
\begin{eqnarray}
&& \lim_{\epsilon\rightarrow 0^+}h(x+i\,\epsilon)=\left(\sqrt{x}\right)^5\sqrt{Q(x)}\, , \nonumber\\
&& \lim_{\epsilon\rightarrow 0^-}h(x+i\,\epsilon)=-\left(\sqrt{x}\right)^5\sqrt{Q(x)}\, . \nonumber
\end{eqnarray}

Now, we define the curve $\sigma$ such that it surrounds all the singularities of the function $f(z)$, \eqref{f(z)}, and having the shape of a big enough square being sufficiently close to the imaginary axis (see the green curve in Figure \ref{path}). We are going to denote by $\sigma_\omega$ the circle of radius $\epsilon$ around $\omega$, where $\omega$ may be either $u_o=\frac{1}{2}\left(\frac{\rho}{2d}\right)^{1/3}$, $\omega_1$, $\bar{\omega}_1$, $\omega_2$ or $\bar{\omega}_2$ (see the blue paths in Figure \ref{path}). Finally, $\sigma_*$ is the curve that surrounds $\beta$ and $\alpha$ such that it is formed by two parts of circles of radius $\epsilon$ centered at $\beta$ and $\alpha$, respectively; together with the segments joining them (see the red curve in Figure \ref{path}). We can assume that all the curves are positively oriented. Then, due to previous limits it is easy to check that 
\begin{equation}
I(d)=-\frac{1}{2}\lim_{\epsilon\rightarrow 0}\int_{\sigma_*} f(z)\,dz\, . \nonumber
\end{equation} 

\begin{figure}[H]
\begin{center}
\includegraphics[width=.5\textwidth]{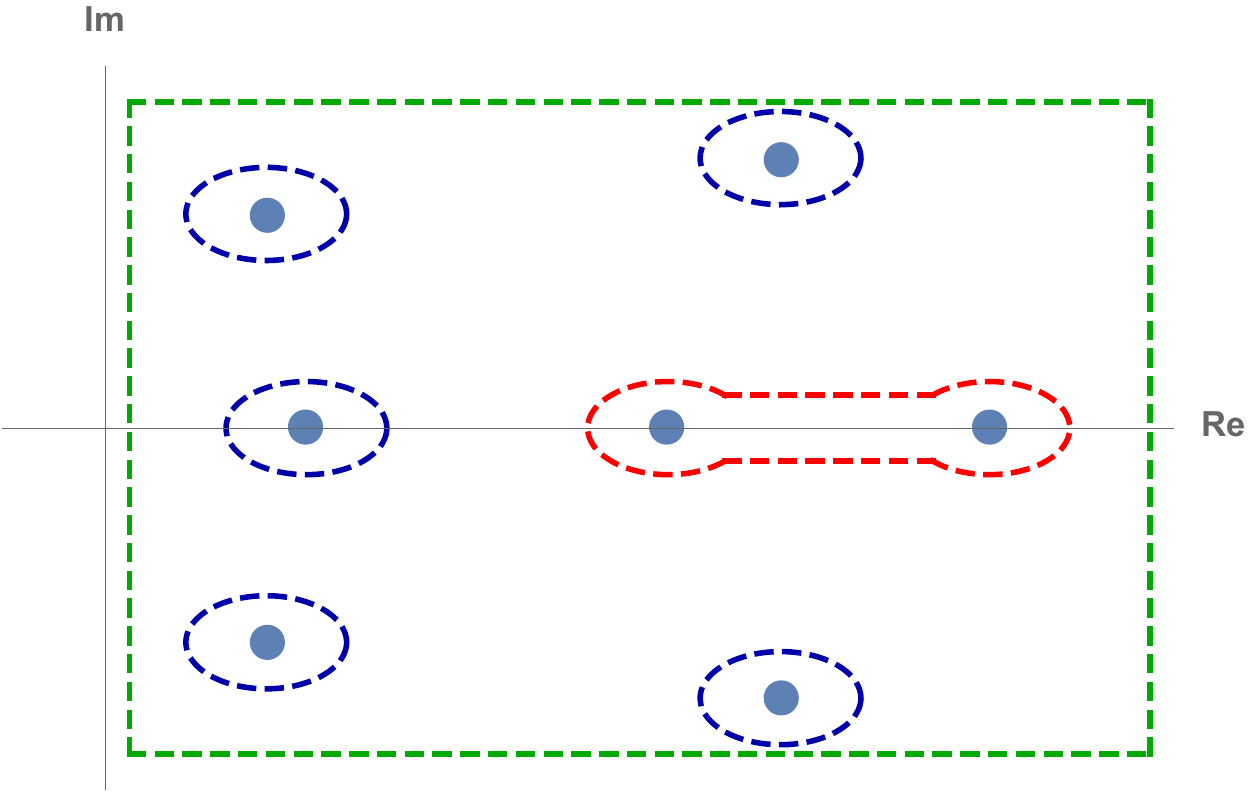}
\end{center}
\caption{Representation of the curves $\sigma$ (in green), $\sigma_\omega$ for each $\omega$ (in blue) and $\sigma_*$ that surrounds the singularities $\beta$ and $\alpha$ (in red).}\label{path}
\end{figure}

If we call $U$ to the region whose boundary is $\sigma$, $\sigma_*^-$ and $\sigma_\omega^-$ for all $\omega$ as above (the enclosed region, see Figure \ref{path}) we have that $f(z)$ is holomorphic in $U$ and, as a consequence,
\begin{equation}
\int_{\sigma} f(z)\,dz=\int_{\sigma_*\cup\sigma_{u_o}\cup\sigma_{\omega_1}\cup\sigma_{\bar{\omega}_1}\cup\sigma_{\omega_2}\cup\sigma_{\bar{\omega}_2}} f(z)\,dz\, . \nonumber
\end{equation}
Moreover, denoting $g(z)=(z-u_o)\,f(z)$ we have that in the region surrounded by $\sigma_{u_o}$, $g(z)$ is analytic. Thus, we apply Cauchy's Integral Formula to compute
\begin{equation}
\int_{\sigma_{u_o}} f(z)\,dz=2\,\pi\,i\left(\frac{1}{2\,\pi\,i} \int_{\sigma_{u_o}} \frac{g(z)}{z-u_o}\, dz\right)=2\,\pi\,i\, g(u_o)=2\,\pi\,i\, \Res_{z=u_o}\, f(z)=2\,\pi\, .\nonumber
\end{equation}
Furthermore, arguing similarly we can check that the sum of the following path integrals vanishes
\begin{equation}
\int_{\sigma_{\omega_2}} f(z)\,dz+\int_{\sigma_{\bar{\omega}_2}} f(z)\,dz=2\,\pi\,i\,\left(\Res_{z=\omega_2}\,f(z)+\Res_{z=\bar{\omega}_2}\,f(z)\right)=0\, .\nonumber
\end{equation}
On the other hand, by applying the Cauchy's Integral Formula once more, we get
\begin{equation}
\int_{\sigma_{\omega_1}}f(z)\,dz=\int_{\sigma_{\bar{\omega}_1}}f(z)\,dz=0\,.\nonumber
\end{equation}
That is, we conclude that
\begin{equation}
\int_{\sigma_*}f(z)\,dz=\int_{\sigma}f(z)\,dz-2\,\pi\,.\nonumber
\end{equation}
Finally, observe that along $\sigma$, $f(z)\rightarrow 0$ whenever $d$ goes to infinity, therefore,
\begin{equation}
\lim_{d\rightarrow \infty}\int_{\sigma_*} f(z)\,dz=\lim_{d\rightarrow\infty} \left(\int_{\sigma}f(z)\,dz-2\,\pi\right)=\int_{\sigma}\lim_{d\rightarrow\infty} f(z)\,dz-2\,\pi=-2\,\pi\, .\nonumber
\end{equation}
Then, considering $\epsilon$ going to zero we get,
\begin{equation}
\lim_{\epsilon\rightarrow 0}\lim_{d\rightarrow\infty}\int_{\sigma_*}f(z)\,dz=-2\,\lim_{d\rightarrow\infty} \,I(d)=-2\,\pi\, . \nonumber
\end{equation}
That is, $\lim_{d\rightarrow\infty}\, I(d)=\pi$, which finishes the second part of the proof.

\subsection{Part (iii)}

Finally, in this last part, we will prove that the function $I(d)=\sqrt{\rho\,d}\,\Lambda(d)$ is monotonically decreasing on $d$. Let us consider the extension to the complex plane introduced in part (ii), \eqref{f(z)}. Then, we know that
\begin{equation}
I(d)=-\frac{1}{2}\lim_{\epsilon\rightarrow 0}\int_{\sigma_*} f(z)\,dz=-\frac{1}{2}\lim_{\epsilon\rightarrow 0}\left(\int_{\sigma}f(z)\,dz-2\,\pi\right).\nonumber
\end{equation}
Thus, if we differentiate above equation we get,
\begin{equation}
I'(d)=-\frac{1}{2}\lim_{\epsilon\rightarrow 0}\left(\frac{\partial}{\partial d}\int_{\sigma}f(z)\,dz\right)=-\frac{1}{2}\lim_{\epsilon\rightarrow 0}\int_{\sigma}f_d(z)\,dz\nonumber
\end{equation}
where
\begin{equation}
f_d(z)=\frac{\partial f}{\partial d}(z)=\frac{18\sqrt{\rho}\,z^5\left(16\,d\,z^6\left(9z-32d\right)+\rho\,z^3\left(16d+9z\right)+\rho^2\right)}{\sqrt{d}\left(16\,d\,z^3-\rho\right)^2 h(z)\left(-9z^4+16\,d\,z^3-\rho\right)}\, . \nonumber
\end{equation}
Moreover, by a similar argument to that of part (ii) and using Cauchy's Integral Formula again we have that
\begin{equation}
\int_{\sigma}f_d(z)\,dz=\int_{\sigma_*}f_d(z)\,dz\,.\nonumber
\end{equation}
That is, combining everything, we obtain that
\begin{equation}
I'(d)=-\frac{1}{2}\lim_{\epsilon\rightarrow 0}\int_{\sigma}f_d(z)\,dz=-\frac{1}{2}\lim_{\epsilon\rightarrow 0}\int_{\sigma_*}f_d(z)\,dz=\int_\beta^\alpha f_d(u)\, du<0\, ,\nonumber
\end{equation}
where last inequality comes from the fact that
\begin{eqnarray}
16 d\left(32d-9u\right)u^6-\rho\left(16d+9u\right)u^3-\rho^2>16d\left(32d-9u\right)u^6-32\rho\, du^3&=&\nonumber\\=16d u^3\left(-9u^4+32 d u^3-2\rho\right)>144d u^7&>&0\,.\nonumber
\end{eqnarray}
That is, $I(d)=\sqrt{\rho\,d}\,\Lambda(d)$ decreases monotonically.

In conclusion, combining parts (i) to (iii) we have that the function $I(d)=\sqrt{\rho\,d}\,\Lambda(d)$ monotonically decreases from $\sqrt{2}\pi$ (obtained when $d\rightarrow d_*$) to $\pi$ (which corresponds with $d\rightarrow\infty$). This concludes the proof of Lemma \ref{lemma}.
\\

\noindent{\bf Acknowledgments.} The authors would like to thank the referees for their valuable comments which have helped to improve the manuscript.


\begin{thebibliography}{777}

\bibitem{Arroyo-Garay-Mencia} J. Arroyo, O. J. Garay and J. J. Menc\'ia. A note on closed generalized elastic curves in $\mathbb{S}^2(1)$. \textit{J. Geom. Phys.} \textbf{48} (2003), 339--353.

\bibitem{Arroyo-Garay-Pampano} J. Arroyo, O. J. Garay and A. P\'ampano. Binormal motion of curves with constant torsion in 3-spaces. \textit{Adv. Math. Phys.} \textbf{2017} (2017), Art. ID 7075831, pp. 8.

\bibitem{Baird-Eells} P. Baird and J. Eells. A conservation law for harmonic maps. Geometry Symposium Utrecht 1980, 1-25, \textit{Lecture Notes in Mathematics  \textbf{894},} Springer, Berlin-New York, 1981.

\bibitem{Baird-Ratto} P. Baird and A. Ratto. Conservation laws, equivariant harmonic maps and harmonic morphisms. \textit{Proc. London Math. Soc.} \textbf{64} (1992), 197--224.

\bibitem{Barros-Garay} M. Barros and O. J. Garay. Critical curves for the normal curvature in surfaces of 3-dimensional space forms. \textit{J. Math. Anal. and App.} \textbf{389} (2012), 275--292.

\bibitem{BMO13} A. Balmu\c s, S. Montaldo and C. Oniciuc. Biharmonic PNMC submanifolds in spheres. {\em Ark. Mat.} \textbf{51} (2013), 197--221.

\bibitem{Caddeo-Montaldo-Oniciuc-Piu} R. Caddeo, S. Montaldo, C. Oniciuc and P. Piu. Surfaces in three-dimensional space forms with divergence-free stress-bienergy tensor. \textit{Ann. Mat. Pura Appl.} \textbf{193} (2014), 529--550.

\bibitem{C84} B-Y. Chen. {\em Total Mean Curvature and Submanifolds of Finite Type.} Series in Pure Mathematics~1. World Scientific Publishing Co., Singapore, 1984.

\bibitem{Eells-Sampson} J. Eells and J. H. Sampson. Harmonic mappings of Riemannian manifolds. \textit{Amer. J. Math.} \textbf{86} (1964), 109--160.

\bibitem{FLO} D. Fetcu, E. Loubeau and C. Oniciuc. Bochner-Simons formulas and the rigidity of biharmonic submanifolds. \textit{J. Geom. Anal.} (2019). https://doi.org/10.1007/s12220-019-00323-y.

\bibitem{FNO} D. Fetcu, S. Nistor and C. Oniciuc. On biconservative surfaces in 3-dimensional space forms. \textit{Comm. Anal. Geom.} \textbf{24} (2016), 1027--1045.

\bibitem{FOP}  D.~Fetcu, C.~Oniciuc and A. L.~Pinheiro. CMC biconservative surfaces in $\s^n\times \r$ and $\h^n\times \r$. \textit{J. Math. Anal. Appl.} {\bf 425} (2015), 588--609.

\bibitem{YuFu} Y.~Fu. On bi-conservative surfaces in Minkowski 3-spaces.  \textit{J. Geom. Phys.} {\bf 66} (2013), 71--79.

\bibitem{Fuannali} Y.~Fu. Explicit classification of biconservative surfaces in Lorentz 3-space forms. \textit{Ann. Mat. Pura Appl.} {\bf 194} (2015), 805--822 .

\bibitem{Fu-Li} Y. Fu and L. Li. A class of Weingarten surfaces in Euclidean 3-Space. \textit{Abs. App. Anal.} \textbf{2013} (2013).

\bibitem{Garay-Pampano} O. J. Garay and A. P\'ampano. A note on $p$-elasticae and the generalized EMP equation. Preprint.

\bibitem{Hilbert} D. Hilbert. Die grundlagen der physik. \textit{Math. Ann.} \textbf{92} (1924), 1--32.

\bibitem{Jiang} G. Y. Jiang. The conservative law for 2-harmonic maps between Riemannian manifolds. \textit{Acta Math. Sinica} \textbf{30} (1987), 220--225.

\bibitem{Jiang-1} G. Y. Jiang. $2$-harmonic maps and their first and second variational formulas. \textit{Chinese Ann. Math. Ser. A} \textbf{7} (1986), 389--402.

\bibitem{Langer-Singer} J. Langer and D. Singer. The total squared curvature of closed curves. \textit{J. Diff. Geom.} \textbf{20} (1984), 1--22.

\bibitem{Lopez-Pampano} R. L\'opez and A. P\'ampano. Classification of rotational surfaces in Euclidean space satisfying a linear relation between their principal curvatures. \textit{Math. Nachr.} \textbf{293} (2020), 735--753.

\bibitem{LM08} E.~Loubeau and S.~Montaldo. Biminimal immersions. {\em Proc. Edinb. Math. Soc.} \textbf{51} (2008), 421--437.

\bibitem{Loubeau-Montaldo-Oniciuc} E. Lobeau, S. Montaldo and C. Oniciuc. The stress-energy tensor for biharmonic maps. \textit{Math. Z.} \textbf{259} (2008), 503--524.

\bibitem{MOR} S.~Montaldo, C.~Oniciuc and A.~Ratto. Proper biconservative immersions into the Euclidean space. \textit{Ann. Mat. Pur. Appl.}  {\bf 195} (2016), 403--422.

\bibitem{MOR2}  S.~Montaldo, C.~Oniciuc and A.~Ratto. Biconservative surfaces. \textit{J. Geom. Anal.}  {\bf 26} (2016), 313--329.

\bibitem{Montaldo-Onnis-Passamani} S. Montaldo, I. Onnis and A. P. Passamani. Biconservative surfaces in BCV spaces. \textit{Math. Nachr.} \textbf{290} (2017), 2661--2672.

\bibitem{Nistor} S. Nistor. Complete biconservative surfaces in $\mathbb{R}^3$ and $\mathbb{S}^3$. \textit{J. Geom. Phys.} \textbf{110} (2016), 130--153.

\bibitem{Nistor-Oniciuc} S. Nistor and C. Oniciuc. Global properties of biconservative surfaces in $\mathbb{R}^3$ and $\mathbb{S}^3$. Proceedings of The International Workshop on Theory of Submanifolds, Istanbul, Turkey, Vol. 1, 2016, pp. 30--56.

\bibitem{Nistor-Oniciuc-2} S. Nistor and C. Oniciuc. On the uniqueness of complete biconservative surfaces in $\mathbb{R}^3$. \textit{Proc. Amer. Math. Soc.} \textbf{147} (2019), 1231--1245.

\bibitem{Nistor-Oniciuc-3} S. Nistor and C. Oniciuc. Complete biconservative surfaces in the hyperbolic space $\mathbb{H}^3$. {\em Nonlinear Anal.} \textbf{198} (2020), 111860, 29 pp. 


\bibitem{O10} Y. -L.~Ou. Biharmonic hypersurfaces in Riemannian manifolds. {\em Pacific J. Math.} \textbf{248} (2010), 217--232.

\bibitem{Perdomo} O. M. Perdomo. Embedded constant mean curvature hypersurfaces on spheres. \textit{Asian J. Math.} \textbf{14} (2010), 73--108.

\bibitem{Sanini} A. Sanini. Applicazioni tra varieta Riemanniane con energia critica rispetto a deformazioni di metriche. \textit{Rend. Mat.} \textbf{3} (1983), 53--63.

\bibitem{Weingarten} J. Weingarten. Ueber eine klasse auf einander abwickelbarer flachen. \textit{J. Reine Angew. Math.} \textbf{59} (1861), 382--393. 

\end{thebibliography}
\end{document}